\documentclass{amsart}
\usepackage{enumitem,amssymb}
\usepackage[margin=1in, top=.8in]{geometry}
\usepackage[all]{xy}
\usepackage{xcolor}

\usepackage{amsfonts, amsmath, amsthm}

\newtheorem{thm}{Theorem}[section]

\newcommand{\Gal}{{\operatorname{Gal}}}

\title{A counterexample to an optimistic guess about \'etale local systems}

\begin{document}

\author{Brian Lawrence}
\author{Shizhang Li}
\maketitle

Relative $p$-adic Hodge theory aims at extending the known $p$-adic Hodge theory to a $p$-adic local system on a rigid variety. 
Let $X$ be a geometrically connected, quasi-compact rigid analytic variety over a $p$-adic field $K$ and let $\mathbb{E}$ be a 
$\mathbb{Q}_p$-local system on the \'{e}tale site $X_{\acute{e}t}$. 
Liu and Zhu~\cite{LZ17} showed that if at one point $\bar{x} \in X(\bar{K})$ 
the stalk $\mathbb{E}_{\bar{x}}$ is de Rham as a $p$-adic Galois representation, then $\mathbb{E}$ is a de Rham local system;
in particular, the stalk of $\mathbb{E}_{\bar{y}}$ at any point $\bar{y} \in X(\bar{K})$ would be de Rham as well. 
They noted that the similar statements replacing ``de Rham'' by either ``crystalline'' or ``semistable'' are both wrong. 
However, inspired by potential semistability of de Rham representations~\cite{Ber02}, 
they ask~\cite[Remark 1.4]{LZ17} if a de Rham local system 
$\mathbb{E}$ on $X$ would become 
semistable\footnote{Here we make an ad hoc definition for a local system $\mathbb{E}$ on $X_{\acute{e}t}$ to be semistable: its stalk $\mathbb{E}_{\bar{y}}$ at any point $\bar{y} \in X(\bar{K})$ is semistable as a $p$-adic Galois representation. This is the weakest requirement, any reasonable definition should imply our condition.} 
after pulling the system back to a finite \'{e}tale cover of $X$, or even after enlarging the ground field $K$ by a finite extension. 
While the former guess may well be true, in this paper we construct an example illustrating the failure of the latter.

\begin{thm}
There exists a projective variety $X$ over a $p$-adic field $K$, and an \'etale local system $\mathbb{E}$ on $X$,
such that $\mathbb{E}$ is de Rham, but for every finite extension $K'/K$, the restriction of $\mathbb{E}$ to $X_{K'}$ is not semistable.

More precisely,
for every finite extension $K'$ of $K$,
there exist a further extension $L/K'$ and a point $x \in X(L)$,
such that $\mathbb{E}_x$ is not semistable as a representation of $\Gal_L$.
\end{thm}

\begin{proof}



Let $p$ be an odd prime, and
let $X$ be the elliptic curve over $\mathbb{Q}_p$ with Weierstrass equation
\[ y^2 = x (x-p) (x+1). \]
Let $X'$ be (the normalization of) the double cover of $X$ given by
\[ y^2 = x (x-p) (x+1), z^2 = x; \]
it is again an elliptic curve, Galois over $X$ with deck transformation group $\mathbf{Z} / 2 \mathbf{Z}$ cyclic of order 2.

Let $\pi \colon X' \rightarrow X$ be the covering map. 
Let $\mathbb{E}$ be the rank-1 direct summand of the rank-2 \'etale sheaf $\pi_* \underline{\mathbf{Q}_p}$ on $X$,
corresponding to the nontrivial character of the deck transformation group $\mathbf{Z} / 2 \mathbf{Z}$.
This $\mathbb{E}$ is de Rham because it comes from geometry.

On the other hand, for any $L$ of absolute ramification degree at least $3$, we can find a point $(x, y)$ above which $\mathbb{E}$ is not semistable, as follows.
Let $x$ be any uniformizer for $L$; then $x (x-p) (x+1)$ is a square in $L$, so we can find $y \in L$ such that $(x, y) \in X(L)$.

Now $\mathbb{E}_{(x, y)}$ is the rank-one representation of $\Gal_L$
\[ \rho_{L' / L} \colon \Gal_L \rightarrow \{ \pm 1\} \subseteq GL_1(\mathbf{Q}_p) \]
having kernel $\Gal_{L'}$, where $L' = L [ \sqrt{x}]$ is a ramified extension.
By ~\cite[Prop.\ 7.17]{FO}, $\rho_{L' / L}$ is not semistable. 
\end{proof}

We would like to thank Bhargav Bhatt, Zarathustra Brady, Johan de Jong, and Xinwen Zhu for interesting discussions. B.\ L.\ would like to thank the University of Michigan for its hospitality, and to acknowledge support from the National Science Foundation.

\bibliographystyle{amsalpha}

\end{document}